\documentclass[a4paper,10pt,reqno]{amsart}
\usepackage{amsmath,amsthm,amsfonts,amssymb,amsopn,color}
\usepackage[pagewise]{lineno}
\usepackage[all,arc]{xy}
\usepackage{esint,enumerate}
\usepackage{mathrsfs}
\usepackage{amsthm}
\usepackage{lipsum}
\usepackage{pgf,tikz}
\usetikzlibrary{arrows}
\usepackage{graphicx,mfpic}
\usepackage{subfigure}
\usepackage[utf8]{inputenc}
\usepackage{hyperref}

\usepackage[utf8]{inputenc}
\usepackage[english]{babel}

\usepackage{etoolbox}
\patchcmd{\section}{\scshape}{\Large\bfseries}{}{}
\makeatletter
\renewcommand{\@secnumfont}{\bfseries}
\makeatother
\newtheorem{theorem}{Theorem}[section]
\newtheorem{lemma}[theorem]{Lemma}

\newtheorem{corollary}[theorem]{Corollary}

\newtheorem*{thma}{Theorem A}
\newtheorem*{thmb}{Theorem B}
\newtheorem*{thmc}{Theorem C}

\theoremstyle{definition}

\newtheorem{remark}[theorem]{Remark}

\newcommand{\rone}{\mathbb{R}}
\newcommand{\rtwo}{{\mathbb R^2}}

\newcommand{\cpx}{\mathbb C}

\newcommand{\calX}{{\mathcal X}}
\newcommand{\calY}{{\mathcal Y}}
\newcommand{\tilu}{{\widetilde u}}
\newcommand{\tilv}{{\widetilde v}}

\DeclareMathOperator{\argmin}{argmin}

\newcommand{\al}{{\alpha}}

\newcommand{\ga}{{\gamma}}

\newcommand{\ve}{{\varepsilon}}
\newcommand{\vp}{{\varphi}}

\newcommand{\Om}{{\Omega}}

\newcommand{\pa}{\partial}

\newcommand{\qand}{{\quad \mbox{and} \quad}}

\newcommand{\qfor}{{\quad \mbox{for} \quad}}
\newcommand{\qie}{{\quad \mbox{i.e.}, \quad}}

\newcommand{\qin}{{\quad \mbox{in} \quad}}
\newcommand{\qon}{{\quad \mbox{on} \quad}}

\newcommand{\qwhere}{{\quad \mbox{where} \quad}}
\newcommand{\qwith}{{\quad \mbox{with} \quad}}

\numberwithin{equation}{section}

\title[On the Convergence of Solutions for the Ginzburg-Landau Equation and System]{On the Convergence of Solutions for the Ginzburg-Landau Equation and System}

\author[R. Hadiji]{Rejeb Hadiji}\thanks{Univ Paris Est Creteil, CNRS, LAMA, F-94010 Creteil, France.
Univ Gustave Eiffel, LAMA, F-77447 Marne-la-Vallee, France.
 email: rejeb.hadiji@u-pec.fr}

 \author[J. Han]{Jongmin Han}\thanks{Department of Mathematics, Kyung Hee University,
 Seoul, 130-701, Korea.
 email:  jmhan@khu.ac.kr}

 \subjclass[2000]{35B40, 35J60, 35Q60.}

\keywords{two component Ginzburg-Landau equations, non-symmetric potential, asymptotic behavior of solutions}

\begin{document}

\begin{abstract}
Let $(u_\varepsilon)$ be a family of solutions of the Ginzburg--Landau equation 
with boundary condition $u_\varepsilon = g$ on $\partial \Omega$ and of degree $0$. 
Let $u_0$ denote the harmonic map satisfying $u_0 = g$ on $\partial \Omega$. 
We show that, if there exists a constant $C_1 > 0$ such that  for $\varepsilon$ sufficiently small we have
$\frac{1}{2} \int_\Omega |\nabla u_\ve|^2 dx \leq C_1 \leq
\frac{1}{2} \int_\Omega |\nabla u_0|^2 dx,$
then
$C_1 = \frac{1}{2} \int_\Omega |\nabla u_0|^2 dx$
and
 $u_\ve  ~\to ~   u_0 \qin H^1(\Om)$.
 We also prove that if  there is a constant $C_2$ such that  for $\ve$ small enough we have
$ \frac12 \int_\Om |\nabla u_\ve|^2 dx \geq C_2 > \frac12 \int_\Om |\nabla u_0|^2 dx,$
then
$|u_{\ve}|$ does not converge uniformly to $1$ on    $\overline{\Om} $.
We   obtain  analogous results for both symmetric and non-symmetric two-component  Ginzburg--Landau systems.
\end{abstract}
\maketitle

%%%%%%%%%%%%%%%%%%%%% manuscript version
%\small
%\noindent{\bf Version. 280830}

\normalsize

%%%%%%%%%%%%%%%%%%%%%%%%%%%%%%%%%%%%%%%%%%%%%%%%%%%%%%%%%%%%%%%%%%%%%%%%%%
%%%%%%%%%%%%%%%%%%%%%%%%%%%%%%%%%%%%%%%%%%%%%%%%%%%%%%%%%%%%%%%%%%%%%%%%%%%%
\section{Introduction}

Let $\Om \subset \rtwo$ be a smooth bounded     domain.
Let
\[ g  : \pa \Om \to S^1=\{ z \in \cpx : |z|=1\}
\]
be a smooth map that has a nonnegative integer-valued  degree $  \deg (g,\pa\Om)=d$.
Let us define
\[ H^1_g (\Om)  = \big\{ u \in   H^1(\Om; \cpx) : u=g  \mbox{ on } \pa \Om \big\}.
\]
For $\ve>0$, we consider the  Ginzburg--Landau   energy functional
\begin{equation}
\label{eq:ftnal Eb}
G_{\ve}  (u ) = \frac{1}{2} \int_\Om |\nabla u|^2  dx + \frac{1}{4\ve^2} \int_\Om (1-|u|^2)^2 dx.
\end{equation}
The Euler-Lagrange equations for $G_{\ve}$ are  the Ginzburg--Landau equations
\begin{equation}
\label{eq:GL}
\left\{
\begin{aligned}
-\Delta u& = \frac{1}{\ve^2} u (1-|u|^2 ) \qin \Omega, \\
u&=g \qquad  \qquad \quad \qon \pa\Omega.
\end{aligned}
\right.
\end{equation}
In \cite{BBH93,BBH94}, Bethuel, Brezis and H\'{e}lein studied the convergence of minimizers.
In particular, if $\deg (g, \pa \Om)=0$, they proved the following:

%%%%%%%%%%%%%%%%%
\begin{thma}{\rm \cite{BBH93}}\label{thma:BBH}
Let $u_{\ve}$ be a minimizer of $G_{\ve}$ over $H^1_g (\Om)$ where $\Om$ is  a star-shaped domain. 
If $d=0$, then $u_{\ve} \to u_{0}$ in $C^k_{loc} (\Om)$ for any nonnegative integer $k$ as $\ve \to 0$ such that $u_0$ is a unique solution of
\begin{equation}
\label{eq:u0 ftnal}
u_0 = \underset{u \in H^1_g (\Om;S^1)}{\argmin} J_g(u) \qwhere  J_g(u) = \frac{1}{2}\int_\Om | \nabla u |^2dx .
\end{equation}
The function $u_0$ satisfies
\begin{equation}
\label{eq:u0 GL}\left\{
\begin{aligned}
-\Delta u& = u |\nabla u|^2 \quad \text{on}\quad \Om,\\
u&=g  \qquad~\qon \pa\Omega, \\
|u|&=1\qquad ~ \qon \Om.
\end{aligned}
\right.
\end{equation}
\end{thma}
%%%%%%%%%%%%%%%%%%%%%%%%%%%%%%%%
\medskip

See also~\cite{BBH94} for the nonzero-degree case,~\cite{HaSh06} for a potential having a zero of infinite order, and~\cite{BMR94} for the quantization effect on the whole plane.  
According to~\cite[Remark~A.1]{BBH94}, the conclusion of Theorem~A can still hold even when $u_\ve$ is not a minimizer.  
Indeed, we have the following.

\smallskip
%%%%%%%%%%%%%%%%%%%%%%%%%%
\begin{thmb}{\rm \cite[p.144]{BBH94}}\label{thmb:BBH-Rmk}
Assume $\deg (g,\pa \Om)=0$ and let  $u_{\ve}$ be a solution of  \eqref{eq:GL}.
If
\begin{equation}\label{eq:u to u0 in H1}
u_\ve \to u_0 \qin  H^1(\Om),
\end{equation}
then the conclusion of Theorem A is valid.
\end{thmb}
%%%%%%%%%%%%%%%%%%%%%%%%%%%%%%%%%%%%

Theorem B tells us that the strong convergence \eqref{eq:u to u0 in H1} is a key ingredient in the proof of Theorem A.

Let $(u_\varepsilon)$ be a sequence of solutions to \eqref{eq:GL}. 
In this work, we establish that 
\[
\frac{1}{2} \int_\Omega |\nabla u_\varepsilon|^2 \, dx 
\]
admits the critical lower bound 
\[
\frac{1}{2} \int_\Omega |\nabla u_0|^2 \, dx,
\]
beyond which the sequence $(u_\varepsilon)$ cannot be lifted to a smooth function, see the proof of Theorem 1.1.

%%%%%%%%%%%%%%%%%
%%%%%%%%%%%%%%%%%%%%

We provide another sufficient condition for Theorem A by identifying an equivalent formulation of \eqref{eq:u to u0 in H1}.
We also introduce a two-component generalization of \eqref{eq:ftnal Eb} and \eqref{eq:GL}, from which we derive analogous results.

Two facts used in the proof of Theorem A will also play a central role in this paper.

%%%%%%%%%%%%%%%%

 First,  if $u_\ve$ is a solution of \eqref{eq:GL}, then
\begin{equation}
\label{u-veleq-1}
 | u_\ve|   \leq 1 \qon \overline{\Om}.
 \end{equation}
  We can prove the inequality  \eqref{u-veleq-1} by applying the maximum principle to the following identity:
 \begin{equation}\label{eq:1-u^2 eqn}
  - \Delta (1- |u_\ve|^2) = - \frac{2}{\ve^2} |u_\ve|^2 (1-|u_\ve |^2) + 2 |\nabla u_\ve|^2 \qon \Om.
\end{equation}
See \cite[Proposition 2]{BBH93}.

Second, if the domain $\Om$ is star-shaped, then for any solution $u_\varepsilon$ of \eqref{eq:GL},  the potential 
\[
\frac{1}{\varepsilon^{2}} \int_\Omega \bigl(1 - |u_\varepsilon|^{2}\bigr)^{2} \, dx 
\]
is bounded.
See \cite[Theorem III.2]{BBH94} and \cite{Str94}.
Moreover,   it is proved in \cite{Lin97} (see also \cite{DelP-F98}) that 
the potential is also bounded  provided that 
\begin{equation}\label{E-bded}
G_\varepsilon(u_\varepsilon) \leq k \ln\frac{1}{\varepsilon}
\end{equation}
for some constant $k>0$. 

In what follows, we suppose that \eqref{E-bded} is valid or $\Om$ is star-shaped. We have then
\begin{equation}\label{potential-bded}
 \frac{1}{\ve^2}  \int_\Om  \big( 1- |u_\ve|^2 \big)^2dx   \le  \ga_0.
 \end{equation}
 Here,   $\ga_0$ depends only on $\Om$ and $g$. 
The first result of this paper is the following theorem.

\smallskip
%%%%%%%%%%%%%%%%%%%%%%%%%%%%%%%%
\begin{theorem}\label{thm:main1}
Suppose that
\begin{equation}\label{eq:deg zero}
\deg (g,\partial \Omega)=0.
\end{equation}
Let $u_\ve$ be a solution of \eqref{eq:GL}.
\begin{itemize}
\item[{\rm (i)}]
If there exists a constant $C_1$ such that, for $\ve$ small enough, we have
\begin{equation}\label{G < C_1}
\frac{1}{2} \int_\Omega |\nabla u_\ve|^2 \, dx \leq C_1 \leq
\frac{1}{2} \int_\Omega |\nabla u_0|^2 \, dx,
\end{equation}
then
\begin{equation}\label{eq:C1-1}
C_1 = \frac{1}{2} \int_\Omega |\nabla u_0|^2 \, dx
\end{equation}
and
\[
u_\ve \to u_0 \quad \text{in } H^1(\Omega).
\]
Thus, Theorem~A holds true by Theorem~B.

\item[{\rm (ii)}]
If there exists a constant $C_2$ such that, for $\ve$ small enough, we have
\begin{equation}\label{G bdd by C1-C2}
\frac{1}{2} \int_\Omega |\nabla u_\ve|^2 \, dx \geq C_2 > \frac{1}{2} \int_\Omega |\nabla u_0|^2 \, dx,
\end{equation}
then
\begin{equation}\label{no-uniform-convergence-2}
|u_{\ve}| \text{ does not converge uniformly to } 1 \text{ on } \overline{\Omega}.
\end{equation}
\end{itemize}

\end{theorem}
%%%%%%%%%%%%%%%%%%%%%%%%%%%%%%%%
\medskip

By using Theorem \ref{thm:main1}, we prove the next theorem where
 we find a condition that is equivalent to \eqref{eq:u to u0 in H1}.

\medskip
%%%%%%%%%%%%%%%%%%%%%%%%%%%%%%%%
\begin{theorem}\label{thm:main2}
Let us assume \eqref{eq:deg zero} and let  $u_\ve$ be a solution for  \eqref{eq:GL}.
Then,
\begin{equation}\label{eq:potential to zero}
\lim_{\ve \to 0} \frac{1}{\ve^2} \int_\Om  \big( 1- |u_\ve|^2 \big)^2 dx=0
\end{equation}
if and only if
\begin{equation}\label{eq:lim nabla u}
   u_\ve  ~\to ~   u_0 \qin H^1(\Om).
\end{equation}
\end{theorem}
%%%%%%%%%%%%%%%%%%%%%%%%%%%%%%%%
\bigskip

As a two-component generalization of \eqref{eq:ftnal Eb}, let us consider
 \begin{equation}
\label{eq:ftnal G2V}
\begin{aligned}
F_\ve (u, v )  = &\frac12 \int_\Om  \big(|\nabla u|^2+|\nabla v|^2\big) \,dx + \frac{1}{4\ve^2}\int_\Om  V(|u|^2,|v|^2) \,dx
 \end{aligned}
\end{equation}
for $(u_\ve,v) \in H^1_{g_1} (\Om) \times H^1_{g_2}(\Om)$.
Here,  $ g_1, g_2  : \pa \Om \to S^1$ are smooth maps such that
\begin{equation}\label{eq:deg of g sys}
d_i=\deg (g_i,\pa\Om)
\end{equation}
 is a nonnegative integer for each $i=1,2$.
 We assume that $\Om$ is star-shaped.
The potential function $V$ is given two cases:
\begin{align*}
\mbox{symmetric case: }& V_s(|u|^2,|v|^2)=(2-|u|^2-|v|^2)^2,\\
\mbox{non-symmetric case: }& V_n(|u|^2,|v|^2)=(2-|u|^2-|v|^2)^2 + (1-|u|^2)^2.
\end{align*}
In each case, $F_\ve$ has a minimizer $(u_\ve,v_\ve)$ over $H^1_{g_1} (\Om) \times H^1_{g_2}(\Om)$.
The potential appears in the semi-local gauge field theories \cite{Hi92,VaAch91}.

The Euler-Lagrange equations are given as follows: for $V=V_s$
\begin{equation}
\label{eq:2-GL sym}
\left\{
\begin{aligned}
-\Delta u & = \frac{1}{\ve^2} u  \big(2-|u |^2-|v |^2\big)     \qin \Omega,\\
-\Delta v & = \frac{1}{\ve^2}v  \big(2-|u  |^2-|v |^2\big) \qin \Omega,\\
u &=g_1, \quad v =g_2 \qquad \quad ~ \qon \pa\Omega,
\end{aligned}
\right.
\end{equation}
and for $V=V_n$
\begin{equation}
\label{eq:2-GL nonsym}
\left\{
\begin{aligned}
-\Delta u & = \frac{1}{\ve^2} u  \big(2-|u |^2-|v |^2\big)  +  \frac{1}{\ve^2} u   (1- |u |^2 )   \qin \Omega,\\
-\Delta v & = \frac{1}{\ve^2}v  \big(2-|u |^2-|v |^2\big)  \qquad\qquad\qquad\quad ~  \qin \Omega,\\
u &=g_1, \quad v =g_2  \qquad\qquad\qquad \qquad \quad\quad ~~  \qon \pa\Omega.
\end{aligned}
\right.
\end{equation}
Now, we want to extend Theorem \ref{thm:main1} for solutions of \eqref{eq:2-GL sym} and \eqref{eq:2-GL nonsym}.
Since \eqref{u-veleq-1} and \eqref{potential-bded} play important roles  in the proof Theorem \ref{thm:main1}, a  natural question arises: can we have inequalities for solutions of  \eqref{eq:2-GL sym} and \eqref{eq:2-GL nonsym} analogous to \eqref{u-veleq-1} and \eqref{potential-bded}?
The answer is not easy.
n fact, although the systems \eqref{eq:2-GL sym} and \eqref{eq:2-GL nonsym} appear to be simple extensions of \eqref{eq:GL}, the nature of their solutions is quite different, as we shall see.

First, one may expect that  if $(u_\ve,v_\ve)$ is a solution of \eqref{eq:2-GL nonsym}, then
\begin{equation}\label{eq:uv less that 1}
|u_\ve|\le 1 \qand |v_\ve|\le 1 \qon \overline{\Om}.
\end{equation}

We recall that \eqref{u-veleq-1} was obtained using the maximum principle applied to the equation \eqref{eq:1-u^2 eqn}.
However, since \eqref{eq:2-GL sym} and \eqref{eq:2-GL nonsym} are systems of equations, it is not possible to derive such an estimate by simply applying the maximum principle.
Instead, weaker versions of \eqref{eq:uv less that 1} were established in \cite{HHS23ANONA,HHS-nonsym-deg-0}

\smallskip
  %%%%%%%%%%%%%%%%%%%%%%%%%%%
\begin{lemma}\cite[Lemma 2.2]{HHS23ANONA}, \cite[Lemma 2.1]{HHS-nonsym-deg-0}
\label{lem:L-infty variant} $\;$
\begin{itemize}
\item[{\rm (i)}]
If  $(u_\ve,v_\ve) $ is a solution pair of  \eqref{eq:2-GL sym}, then we   have
\begin{equation}
\label{eq:L-infty var sym}
|u_\ve|^2 + |v_\ve|^2 \leq 2 \qon \overline\Om.
\end{equation}

\item[{\rm (ii)}]
If  $(u_\ve,v_\ve) $ is a solution pair of  \eqref{eq:2-GL nonsym}, then we   have
\begin{equation}
\label{eq:L-infty var}
|u_\ve|^2  \leq \frac{3}{2} \qand  |v_\ve|^2 \leq 2 \qon \overline\Om.
\end{equation}
Moreover, either $|u_\ve|\le 1$ or $|v_\ve| \le 1$ on $\overline\Om$.
\end{itemize}
\end{lemma}
%%%%%%%%%%%%%%%%%%%%%%%%%%%%%%%%
\smallskip

The first statement (i) gives no information on the indivisual upper bounds of $|u_\ve|$ and $|v_\ve|$ although theirs sums are bounded by $2$.
The second statement provide no information on the bounds of $|u_\ve|^2 + |v_\ve|^2$ and the upper bounds of $|u_\ve|$ and $|v_\ve|$ are rather rough compared to \eqref{eq:uv less that 1}.
Since the pointwise estimate $|u_\ve|\le 1$ for solutions of \eqref{eq:GL} are crucial in various analysis of solutions, it is very interesting to prove \eqref{eq:uv less that 1} or to make analysis of solutions of \eqref{eq:2-GL sym} and \eqref{eq:2-GL nonsym} without appealing the property of \eqref{eq:uv less that 1}.

Second difference among solutions  of  \eqref{eq:GL}, \eqref{eq:2-GL sym} and \eqref{eq:2-GL nonsym} is the Pohozaev identity.
Analogous to \eqref{potential-bded}, we can prove that if $\Om$ is star-shaped, then
\begin{align}
\label{eq:Pohozaev sys sym}
\text{$(u_\ve,v_\ve)$: solution of \eqref{eq:2-GL sym}} ~ & \Longrightarrow ~
 \frac{1}{ \ve^2}\int_\Om   (2-|u_\ve|^2-|v_\ve|^2)^2 dx \le \ga_1 ,\\
\nonumber
\text{$(u_\ve,v_\ve)$: solution of \eqref{eq:2-GL nonsym}} ~ & \Longrightarrow ~
 \frac{1}{ \ve^2}\int_\Om   (2-|u_\ve|^2-|v_\ve|^2)^2 dx  \\
\label{eq:Pohozaev sys nonsym}
 &\qquad \qquad\qquad+ \frac{1}{ \ve^2}\int_\Om (1-|u_\ve|^2)^2 \,dx \le \ga_2
 \end{align}
for some constants $\ga_1$ and $\ga_2$.
Since we do not know the signs of $1-|u_\ve|^2$ and $1-|v_\ve|^2$, \eqref{eq:Pohozaev sys sym} does not imply
\begin{equation}\label{eq:pot u and v}
 \frac{1}{ \ve^2}\int_\Om (1-|u_\ve|^2)^2 \,dx~<~\infty \qand   \frac{1}{ \ve^2}\int_\Om (1-|v_\ve|^2)^2 \,dx~<~\infty .
\end{equation}
Indeed, these quantities can diverge for some solutions of \eqref{eq:2-GL sym} although they satisfy   \eqref{eq:Pohozaev sys sym}.
See Theorem C below.
On the other hand,   solutions of \eqref{eq:2-GL nonsym} always satisfy not only \eqref{eq:Pohozaev sys nonsym} but also \eqref{eq:pot u and v}.
See the proof of Theorem \ref{thm:main4} (ii) below.

 To state the main results on the solutions of  \eqref{eq:2-GL sym} and \eqref{eq:2-GL nonsym}, we assume that $d_1=d_2=0$ in \eqref{eq:deg of g sys} and $\Om$ is star-shaped..
 We set
 \begin{align*}
 \calY(g_1,g_2) &:= H^1_{g_1} (\Om;S^1) \times H^1_{g_2} (\Om;S^2),\\
 \calX(g_1, g_2) & : =\Big\{  (u_ ,v) \in H^1_{g_1}(\Om; \mathbb C) \times   H^1_{g_2} (\Om; \mathbb C) ~:~  |u|^2 + |v|^2= 2~  \text{ a.e. on } \Om\Big\},
 \end{align*}
 and
 \[   I_{(g_1,g_2)} (u , v)   =\frac12  \int_\Om  \big(|\nabla u|^2+|\nabla v|^2  \big)\, dx  = J_{g_1}(u)+J_{g_2} (v).
\]
Let us consider the following minimization problems:
\begin{align}
\label{eq:min prob alpha}
\alpha(g_1,g_2) &:= \inf \big\{ I_{(g_1,g_2)} (u,v) ~:~ (u,v) \in \calY (g_1,g_2)\big\},\\
\label{eq:min prob beta}
\beta(g_1,g_2) &:= \inf \big\{ I_{(g_1,g_2)} (u,v) ~:~ (u,v) \in \calX (g_1,g_2)\big\}.
\end{align}
The problem \eqref{eq:min prob alpha} has   a unique solution  $(u_0,v_0)$ on $\calY(g_1,g_2)$ that satisfies
\[ \left\{
\begin{aligned}
-\Delta u_0& = u_0 |\nabla u|^2 \quad \text{on}\quad \Om,\\
u_0&=g_1 \qquad \qon \pa\Omega, \\
|u_0|&=1\qquad ~~ \qon \Om,
\end{aligned}\right.
\qquad\left\{
\begin{aligned}
-\Delta v& = v |\nabla v|^2 \quad \text{on}\quad \Om,\\
v_0&=g_2 \qquad \qon \pa\Omega, \\
|v_0|&=1 \qquad ~~ \qon \Om.
\end{aligned}\right.
\]
If $ (u_*,v_*)$ is a solution of \eqref{eq:min prob beta}, then $(u_*,v_*)$ satisfies
\[ \left\{
\begin{aligned}
-\Delta u_*& = \frac12 u_* ( |\nabla u_*|^2+ |\nabla v_*|^2) ~~ \text{on}~~ \Om, \quad
u_* =g_1    ~~ \text{on}~~  \pa\Omega, \\
-\Delta v_*& = \frac12 v_* ( |\nabla u_*|^2+ |\nabla v_*|^2) ~~ \text{on}~~ \Om, \quad
v_* =g_2    ~~ \text{on}~~  \pa\Omega, \\
2&=|u_*|^2+|v_*|^2 ~~~ \text{a.e.} \quad  \qon \Om.
\end{aligned}\right.
\]
Since $\calY(g_1,g_2) \subset \calX(g_1,g_2)$, it is obvious that
\begin{equation}\label{eq:alpha >= beta}
\alpha (g_1,g_2) \ge \beta (g_1,g_2).
\end{equation}
The next theorem tells us that  \eqref{eq:alpha >= beta} has a close relation with some properties of solutions of \eqref{eq:2-GL sym}.

\smallskip
%%%%%%%%%%%%%%%%%%%%%%%%%%%%%%%%%%
\begin{thmc}\cite[Theorem 1.3 (iii)]{HHS23ANONA}
Suppose that
\begin{equation}\label{eq:deg zero g1 g2}
\deg (g_1,\pa \Om)= \deg (g_2,\pa \Om)=0.
\end{equation}
Let $(u_\ve,v_\ve)$ be a minimizer of \eqref{eq:ftnal G2V} with $V=V_s$.
If $\al(g_1,g_2)>\beta(g_1,g_2)$, then
\[ \lim_{\ve \to 0} \frac{1}{\ve^2} \int_\Om \big( 1- |u_\ve|^2 \big)^2 dx = \lim_{\ve \to 0} \frac{1}{\ve^2} \int_\Om \big( 1- |v_\ve|^2 \big)^2 dx =\infty.
\]
\end{thmc}
%%%%%%%%%%%%%%%%%%%%%%%%%%%%%%%%%%%%%
\medskip

Now, we extend Theorem \ref{thm:main1} for solutions of  \eqref{eq:2-GL nonsym} as follows.

\medskip
%%%%%%%%%%%%%%%%%%%%%%%%%%%%%%%%
\begin{theorem}\label{thm:main3}
Let $\Om$ be star-shaped.
Suppose that $d_1=d_2=0$ such that  \eqref{eq:deg zero g1 g2} holds.
Let  $(u_\ve,v_\ve)$ be a solution for \eqref{eq:2-GL nonsym} and $(u_0,v_0)$ be a unique minimizer of $I_{(g_1,g_2)}$ on $\calY(g_1,g_2)$.
\begin {itemize}
\item[{\rm (i)}]
If  there is a constant $C_3$  such that  we have for $\ve$ small enough
\begin{equation}\label{G < C_1-sys-nonsym}
\frac{1}{2} \int_\Omega \big( |\nabla u_\ve|^2 + |\nabla v_\ve|^2 \big)\, dx \leq C_3 \leq
\frac{1}{2} \int_\Omega \big( |\nabla u_0|^2 + |\nabla v_0|^2 \big)\, dx ,
\end{equation}
then
\begin{equation}\label{eq:C1-1-sys-nonsym}
C_3 = \int_\Omega \big( |\nabla u_0|^2 + |\nabla v_0|^2 \big)\, dx
\end{equation}
and
\[  (u_\ve,v_\ve)  ~\to ~  ( u_0,v_0) \qin  H^1(\Om)\times H^1(\Om).
\]

\item[{\rm (ii)}]
If  there is a constant $C_4$  such that for $\ve$ small enough we have
\begin{equation}\label{G bdd by C1-C2-sys-nonsym}
 \frac12 \int_\Om \big( |\nabla u_\ve|^2 + |\nabla v_\ve|^2 \big)\, dx \geq C_4 >   \frac12 \int_\Om \big( |\nabla u_0|^2 +  |\nabla v_0|^2 \big) \, dx,
\end{equation}
then
\begin{equation}\label{eq:potential zero nonsym}
\lim_{\ve \to 0} \frac{1}{\ve^2}  \int_\Om  \Big[ \big( 2- |u_\ve|^2 - |v_\ve|^2 \big)^2 +    \big( 1- |u_\ve|^2   \big)^2 \Big] dx > 0.
\end{equation}
\end{itemize}
\end{theorem}
%%%%%%%%%%%%%%%%%%%%%%%%%%%%%%%%
\bigskip

Next, we deal with solutions of \eqref{eq:2-GL sym}.
In view of Theorem C, we obtain the following theorem.

\smallskip
%%%%%%%%%%%%%%%%%%%%%%%%%%%%%%%%
\begin{theorem}\label{thm:main4}
Let $\Om$ be star-shaped.
Suppose that $d_1=d_2=0$ such that  \eqref{eq:deg zero g1 g2} holds.
Let  $(u_\ve,v_\ve)$ be a solution for \eqref{eq:2-GL sym} and $(u_*,v_*)$ be a   minimizer of $I_{(g_1,g_2)}$ on $\calX(g_1,g_2)$.
\begin {itemize}
\item[{\rm (i)}]
If  there is a constant $C_5$  such that  we have for $\ve$ small enough
\begin{equation}\label{G < C_1-sys-sym}
\frac{1}{2} \int_\Omega \big( |\nabla u_\ve|^2 + |\nabla v_\ve|^2 \big)\, dx \leq C_5 \leq
\frac{1}{2} \int_\Omega \big( |\nabla u_*|^2 + |\nabla v_*|^2 \big)\, dx ,
\end{equation}
then
\begin{equation}\label{eq:C1-1-sys-sym}
C_5 = \int_\Omega \big( |\nabla u_*|^2 + |\nabla v_*|^2 \big)\, dx
\end{equation}
and there exists  $( \tilu ,\tilv)   \in \calX(g_1,g_2)$ such that
\[  (u_\ve,v_\ve)  ~\to ~  ( \tilu ,\tilv) \qin H^1(\Om)\times H^1(\Om).
\]
If $\al(g_1,g_2)=\beta(g_1,g_2)$, then $(\tilu,\tilv)=(u_0,v_0)$.

\item[{\rm (ii)}]
Assume that
\begin{align}
\label{eq:al=beta}
 \al(g_1,g_2)& =\beta(g_1,g_2) ,\\
\label{eq:gamma3}
 \lim_{\ve \to 0} \frac{1}{\ve^2} \int_\Om \big(1-|u_\ve|^2\big)^2dx  &\le \ga_3, \\
\label{eq:gamma4}
 \lim_{\ve \to 0} \frac{1}{\ve^2} \int_\Om \big(1-|v_\ve|^2\big)^2dx &  \le \ga_4
\end{align}
If there is a constant $C_6$ such that  for $\ve$ small enough we have
\begin{equation}
\label{G bdd by C1-C2-sys-sym}
\begin{aligned}
&  \frac12 \int_\Om \big( |\nabla u_\ve|^2 + |\nabla v_\ve|^2 \big)\, dx \\
\geq ~& C_6 ~> ~  \frac12 \int_\Om \big( |\nabla u_0|^2 +  |\nabla v_0|^2 \big) \, dx  + \sqrt{\ga_1\ga_3} + \sqrt{\ga_1\ga_4},
\end{aligned}
 \end{equation}
then
\[ \text{either $|u_\ve|$ or $|v_\ve|$ does not converges uniformly to 1 on $\overline\Om$.}
\]
\end{itemize}
\end{theorem}
%%%%%%%%%%%%%%%%%%%%%%%%%%%%%%%%
\bigskip

We will prove Theorem \ref{thm:main1} and  \ref{thm:main2} in Section \ref{sec:thm pf main1 first}.
The proofs of Theorem \ref{thm:main3} and \ref{thm:main4} are given in Section \ref{sec:thm pf main3} and Section \ref{sec:thm pf main4}, respectively.

\bigskip
%%%%%%%%%%%%%%%%%%%%%%%%%%%%%%%%%%
%%%%%%%%%%%%%%%%%%%%%%%%%%%%%%%%%
%%%%%%%%%%%%%%%%%%%%%%%%%%%%%%%%%%%
\section{Proof of Theorem \ref{thm:main1} and Theorem \ref{thm:main2}}\label{sec:thm pf main1 first}

Throughout this section, we assume \eqref{eq:deg zero} and prove Theorem \ref{thm:main1} and Theorem \ref{thm:main2}.
Then, we can write
\[ g = e^{i \vp_0} \qwhere \vp_0:  \pa \Om \to \rone.
\]
Moreover, the function $u_0$ is  lifted by a harmonic function $\varphi$   such that
\[ \left\{
\begin{aligned}
\Delta \vp & =0 \qin \Om \qand \vp=\vp_0 \qon \pa \Om,\\
u_0 &= e^{i \varphi} \qand \int_\Omega |\nabla u_0|^2 dx =\int_\Omega |\nabla \vp|^2 dx.
 \end{aligned}
 \right. \]

 \smallskip
%%%%%%%%%%%%%%
{\bf Proof of Theorem \ref{thm:main1} (i): }
Suppose that \eqref{G < C_1}.
Since $\| u_\ve \|_\infty \leq 1$, up to a subsequence, we have
$u_\ve \rightharpoonup \tilde{u}$ in $H_g^1(\Omega)$ for some $\tilu \in H_g^1(\Omega)$.
By \eqref{potential-bded},  $|\tilu | =1$ a.e. on $\Om$ and consequently $ {\tilde{ u}} \in H^1_g(\Om; S^1)$.
Since $u_0$ is a minimizer of $J_g$, we are led to 
\begin{equation}\label{eq:lim-inf}
\begin{aligned}
\frac{1}{2} \int_\Omega |\nabla u_0|^2 dx & \leq \frac{1}{2} \int_\Omega |\nabla \tilde{u}|^2 dx
\leq \liminf_{\ve \to 0}  \frac{1}{2}\int_\Omega |\nabla u_\ve|^2  dx \\
&\leq C_1 \le   \frac{1}{2}\int |\nabla u_0|^2 dx .
\end{aligned}
\end{equation}
Thus, \eqref{eq:C1-1} is true.
Since $u_\ve \to u_0$ weakly in $H^1(\Om)$, we deduce that
\begin{equation}\label{eq:strong conv u to u0}
\begin{aligned}
& \int_\Om |\nabla u_\ve - \nabla u_0|^2 \, dx \\
=~& \int_\Om |\nabla u_\ve |^2 \, dx + \int_\Om |\nabla u_0 |^2 \, dx -2 \int_\Om \nabla u_\ve \cdot \nabla u_0\, dx ~\to ~ 0.
\end{aligned}
\end{equation}
Hence, $u_\ve \to u_0$ in $H_g^1(\Om)$.
Thus, Theorem A holds true by Theorem B.
\qed \\
%%%%%%%%%%%%%%%%%%%%%%%%%%%%%%%%%%%%%%%%%5
\smallskip

In the above proof, we prove the following corollary.

\smallskip
%%%%%%%%%%%%%%%%%
\begin{corollary}\label{cor:liminf}
If $u_\ve$ is any solution for  \eqref{eq:GL}$_\ve$, then
\[ \liminf_{\ve \to \infty} \int_\Om |\nabla u_\ve|^2  dx \ge  \int_\Om |\nabla u_0|^2  dx.
\]
\end{corollary}
\begin{proof}
If we assume the contrary, up to a subsequence, we may assume that
\[ \frac{1}{2} \int_\Omega |\nabla u_\ve|^2 dx \leq C_1 < \frac{1}{2} \int_\Omega |\nabla u_0|^2 dx
\]
for some $C_1$.
Then, we get a contradiction by arguing as in  \eqref{eq:lim-inf}.
\end{proof}
%%%%%%%%%%%%%%%%%%%%%%%%%%%%%%%%%%%
\smallskip

To prove Theorem \ref{thm:main1} (ii), we need two lemmas.

\smallskip
%%%%%%%%%%%%%%%%%
\begin{lemma}\label{lem:unif conv}
Let  $u_\ve$ be a solution for  \eqref{eq:GL}$_\ve$.
If $|u_\ve| \to 1 $ uniformly on $ \overline\Om $, then
\begin{equation}\label{eq:nabla u L2}
 \int_\Om  |\nabla u_\ve|^2  dx \le 2  \int_\Om  |\nabla u_0|^2  dx .
  \end{equation}
\end{lemma}
\begin{proof}
 Since  $|u_\varepsilon| \to 1$ uniformly on $ \overline\Omega$, we may assume that
\begin{equation}\label{eq:|u| bigger 1/2}
 \text{$|u_\varepsilon| \geq \frac{1}{2}$ on  $\Omega$   for $\varepsilon > 0$ small enough.}
\end{equation}
  Then, $u_\varepsilon / |u_\varepsilon|$ can be lifted by a smooth function $\zeta_\varepsilon$ such that
\[
\frac{u_\varepsilon}{|u_\varepsilon|} = e^{i \zeta_\varepsilon} \qon  \Omega.
\]
Hence, we can write
\[
u_\varepsilon = \rho_\varepsilon e^{i \zeta_\varepsilon} \qwith   \rho_\varepsilon = |u_\varepsilon|.
\]
Then,  $\zeta_\ve = \vp_0$ on $\pa \Om$ and
\begin{equation}\label{eq:nabla form}
 |\nabla u_\varepsilon|^2 =    |\nabla \rho_\varepsilon|^2 +  \rho_\varepsilon^2 |\nabla \zeta_\varepsilon|^2
\end{equation}
and    the equation \eqref{eq:GL} is transformed into a system of $\rho_\ve$ and $\zeta_\ve$:
\begin{align}\label{eq-rho-zeta-1}
\operatorname{div}\left( \rho_\varepsilon^2 \nabla \zeta_\varepsilon \right) &= 0 \qin \Omega, \\
\label{eq-rho-zeta-2}
-\Delta \rho_\ve + \rho_\ve\big| \nabla  \zeta_\varepsilon  \big|^2 & = \frac{1}{\ve^2}\rho_\ve(1- \rho_\ve^2) \qin \Om.
\end{align}
Multiplying \eqref{eq-rho-zeta-1} by $\rho_\ve-1$, we obtain
\begin{equation}\label{eq-rho-2-1}
\begin{aligned}
& \int_\Omega |\nabla \rho_\varepsilon|^2  dx +
  \int_\Omega\rho_\ve^2 |\nabla  \zeta_\varepsilon |^2  dx -   \int_\Omega \rho_\ve  |\nabla  \zeta_\varepsilon |^2   dx \\
=~& \frac{1}{ \ve^2}\int_\Omega \rho_\ve( \rho_\varepsilon - 1)(1- \rho_\ve^2) dx  ~\leq~ 0.
\end{aligned}
\end{equation}
Hence, it comes from \eqref{eq:|u| bigger 1/2}, \eqref{eq:nabla form} and \eqref{eq-rho-2-1} that
\[ \int_\Om |\nabla u_\varepsilon|^2 dx \le  \int_\Omega \rho_\ve  |\nabla  \zeta_\varepsilon |^2   dx \le  2  \int_\Omega \rho_\ve^2  |\nabla  \zeta_\varepsilon |^2   dx.
\]
On the other hand, multiplying \eqref{eq-rho-zeta-1} by $\zeta_\ve -\vp$, we have
\begin{align*}
 \int_\Om \rho_\ve^2 |\nabla \zeta_\ve|^2 dx & =   \int_\Om \rho_\ve^2   \nabla \zeta_\ve \cdot \nabla \vp \, dx \le \left(  \int_\Om \rho_\ve^2 |\nabla \zeta_\ve|^2 dx \right)^{\frac12}  \left(  \int_\Om \rho_\ve^2 |\nabla \vp |^2 dx \right)^{\frac12}.
 \end{align*}
 In this integration, we used the fact   $u_0=u_\varepsilon = g$ on $\partial \Omega$,  i.e., $\vp=\zeta_\ve=\vp_0$ on $\pa \Om$.
 Hence, we conclude that
 \[ \int_\Om |\nabla u_\varepsilon|^2 dx \le 2   \int_\Om \rho_\ve^2 |\nabla \vp |^2 dx \le 2    \int_\Om  |\nabla \vp |^2 dx .
 \qedhere
 \]
\end{proof}
%%%%%%%%%%%%%%%%%%%%%%%%%%%%%%%%%%%%%%

\smallskip
%%%%%%%%%%%%%%%%%
\begin{lemma}\label{lem:limsup}
Let  $u_\ve$ be a solution for  \eqref{eq:GL}$_\ve$.
If $|u_\ve| \to 1 $ uniformly on $\overline\Om $, then
\begin{equation}\label{eq:limsup}
 \limsup_{\ve \to \infty} \int_\Om |\nabla u_\ve|^2  dx \le  \int_\Om |\nabla u_0|^2  dx.
\end{equation}
\end{lemma}
\begin{proof}
Let us assume the contrary.
Then,  there exists a constant $C_2>0$ and a subsequence, still denoted by $u_\ve$, such that
\begin{equation}\label{eq:larger than B1}
\int_\Omega |\nabla \vp|^2 dx=\frac12 \int_\Om |\nabla u_0|^2 dx< C_2 \le  \frac12 \int_\Om |\nabla u_\ve|^2 dx.
\end{equation}
 Since  $|u_\varepsilon| \to 1$ uniformly on $\overline\Omega$, we may keep the notations in the proof of Lemma \ref{lem:unif conv}.
Given   $\delta \in (0,\frac14)$, if $\ve$ is small enough, then
\begin{equation}\label{eq:delta chosen}
\frac{1}{2} < \rho_\ve <\rho_\ve^2 +\delta  \qie \frac{1+\sqrt{1-4\delta} }{2} <\rho_\ve <1.
\end{equation}
Let
\[ \psi_\varepsilon = \zeta_\varepsilon - \varphi.
\]
Then, by \eqref{eq:nabla form} and \eqref{eq:larger than B1},
\begin{equation}\label{eq:nabla u formula}
\begin{aligned}
C_2 \le  \frac{1}{2} \int_\Omega |\nabla u_\varepsilon|^2  dx
&=    \frac{1}{2} \int_\Omega |\nabla \rho_\varepsilon|^2  dx
+\frac{1}{2} \int_\Omega \rho_\varepsilon^2 \big|\nabla (\vp +\psi_\ve) \big|^2 dx.
\end{aligned}
\end{equation}
We rewrite \eqref{eq-rho-zeta-1} and \eqref{eq-rho-zeta-2} as 
\begin{align}\label{eq-rho-psi-1}
\operatorname{div}\left( \rho_\varepsilon^2 \nabla( \varphi + \psi_\varepsilon ) \right) &= 0 \qin \Omega, \\
\label{eq-rho-psi-2}
-\Delta \rho_\ve + \rho_\ve\big| \nabla (\varphi + \psi_\varepsilon )\big|^2 & = \frac{1}{\ve^2}\rho_\ve(1- \rho_\ve^2) \qin \Om.
\end{align}
Multiplying \eqref{eq-rho-psi-2} by $ \rho_\varepsilon - 1 $  and  integrating it over $ \Omega $, and using the boundary
condition $ \rho_\varepsilon  =1 $ on $\partial\Omega$, we obtain
\begin{equation}\label{eq-rho-psi-2-1}
\begin{aligned}
&\frac12\int_\Omega |\nabla \rho_\varepsilon|^2  dx +
\frac12 \int_\Omega\rho_\ve^2 | \nabla (\varphi + \psi_\varepsilon )|^2 dx - \frac12 \int_\Omega \rho_\ve | \nabla (\varphi + \psi_\varepsilon )|^2 dx \\
=~& \frac{1}{2\ve^2}\int_\Omega \rho_\ve( \rho_\varepsilon - 1)(1- \rho_\ve^2) dx  ~\leq~ 0.
\end{aligned}
\end{equation}
Then,  from   \eqref{eq:delta chosen},  \eqref{eq:nabla u formula} and \eqref{eq-rho-psi-2-1}, it follows that
\begin{equation}\label{eq-C2-final}
\begin{aligned}
C_2 ~&\le ~   \frac12
\int_\Omega\rho_\ve \big| \nabla (\varphi + \psi_\varepsilon )\big|^2 dx
 ~ \leq~   \frac12   \int_\Omega (\rho_\varepsilon^2+\delta)  \big| \nabla (\varphi + \psi_\varepsilon )\big|^2  dx \\
&  \le ~   \frac12   \int_\Omega  \rho_\varepsilon^2   \big( | \nabla  \varphi |^2 +2 \nabla  \varphi \cdot  \psi_\varepsilon +|\nabla \psi_\varepsilon|^2  \big)\,  dx + \frac12 \delta    \| \nabla \zeta_\ve \|_2^2 \\
&  \le ~  \frac12   \int_\Omega     | \nabla  \varphi |^2 +\frac12   \int_\Omega  \rho_\varepsilon^2   \big(2 \nabla  \varphi \cdot  \psi_\varepsilon +|\nabla \psi_\varepsilon|^2  \big)\,  dx  +   \frac12 \delta    \| \nabla \zeta_\ve \|_2^2 .
\end{aligned}
\end{equation}
Multiplying \eqref{eq-rho-psi-1} by \( \psi_\varepsilon \), integrating it over \( \Omega \), and using the boundary condition \( \psi_\varepsilon = 0 \) on \( \partial\Omega \), we obtain
\begin{align}\label{eq-rho-psi-1-1}
  \int_\Omega \rho_\varepsilon^2 \left| \nabla \psi_\varepsilon \right|^2 \, dx  +\int_\Omega \rho_\varepsilon^2 \nabla \varphi \cdot \nabla \psi_\varepsilon \, dx =0.
\end{align}
Furthermore, by \eqref{eq:nabla u L2} and \eqref{eq:delta chosen},
\begin{equation}\label{eq:nabla zeta}
   \| \nabla \zeta_\ve \|_2^2 \le 4 \int_\Om \rho_\ve^2 | \nabla \zeta_\ve|^2 dx \le 8 \int_\Om |\nabla u_0|^2 dx.
\end{equation}
Hence, by   \eqref{eq-C2-final},  \eqref{eq-rho-psi-1-1} and \eqref{eq:nabla zeta}, we are led to
\begin{align*}
 0< C_2 -   \frac12   \int_\Omega     | \nabla  \varphi |^2 \le   -   \int_\Omega \rho_\varepsilon^2 \left| \nabla \psi_\varepsilon \right|^2 \, dx+  4 \delta  \|\nabla u_0 \|_2^2 \le   4 \delta   \|\nabla u_0|\|_2^2.
\end{align*}
Letting $\delta \to 0$, we arrive at  a contradiction.
\end{proof}
%%%%%%%%%%%%%%%%%%%%%%%%%%%%%%%%%%%%%%

\smallskip
%%%%%%%%%%%%%%%%%
\begin{lemma}\label{lem:pot to zero implies unif conv}
Let  $u_\ve$ be a solution for  \eqref{eq:GL}$_\ve$ that satisfies \eqref{eq:potential to zero}.
Then, $|u_\ve| \to 1 $ uniformly on $\overline\Om $.
\end{lemma}
\begin{proof}
See \cite[Step A.1, B.2]{BBH93}.
\end{proof}
%%%%%%%%%%%%%%%%%%%%%%%%%%%%%%%%%%%%%%

\smallskip
%%%%%%%%%%%%%%%%%
\begin{lemma}\label{lem:H1 conv implies unif conv}
Let  $u_\ve$ be a solution for  \eqref{eq:GL}$_\ve$.
If  $u\to u_0$ in $H^1(\Om)$, then $|u_\ve| \to 1 $ uniformly on $\overline\Om $.
\end{lemma}
\begin{proof}
 By multiplying \eqref{eq:1-u^2 eqn} by $1-|u_\ve|^2$, we obtain
 \begin{equation}\label{eq:1-u^2 int}
\begin{aligned}
& 2 \int_\Om |\nabla u_\ve|^2 (1-|u_\ve|^2)\, dx\\
=~& \frac{2}{\ve^2} \int_\Om |u_\ve|^2 \big(1-|u_\ve|^2 \big)^2 dx + \int_\Om \big| \nabla (1-|u_\ve|^2 )\big|^2 dx  .
\end{aligned}
\end{equation}
Given $\delta \in (0,\frac14)$, let
\[ \Om_\ve^\delta = \{ x \in \Om :   1-|u_\ve|^2 >\delta \}.
\]
By \eqref{potential-bded},
\[ \ga_0 \ge \frac{1}{\ve^2} \int_{ \Om_\ve^\delta}  \big( 1- |u_\ve|^2\big)^2 dx \ge  \frac{(1-\delta)^2}{\ve^2}  | \Om_\ve^\delta|.
\]
Hence, for all  $\delta \in (0,\frac14)$, $ | \Om_\ve^\delta| \to 0$   as $\ve\to 0$.
Since $u\to u_0$ in $H^1(\Om)$, it follows that for each fixed  $\delta \in (0,\frac14)$,
\[ \int_{ \Om_\ve^\delta}  |\nabla u_\ve|^2 dx \le 2 \int_{ \Om_\ve^\delta}  |\nabla u_\ve - \nabla u_0 |^2 dx + 2  \int_{ \Om_\ve^\delta}  | \nabla u_0 |^2 dx ~ \to ~ 0
\]
as $\ve \to 0$.
Since $u\to u_0$ in $H^1(\Om)$, we have $\|\nabla u_\ve\|_2^2 \le C$ for some $C$.
Now, we see that  as $\ve \to 0$,
\begin{align*}
  \int_\Om |\nabla u_\ve|^2 (1-|u_\ve|^2)\, dx & \le  \delta  \int_{ \Om \backslash \Om_\ve^\delta} |\nabla u_\ve|^2 dx +  \int_{     \Om_\ve^\delta} |\nabla u_\ve|^2 dx \le C \delta  +o(1).
\end{align*}
So, we deduce from \eqref{eq:1-u^2 int} that for all  $\delta \in (0,\frac14)$,
\[ \limsup_{\ve \to 0} \frac{1}{\ve^2} \int_\Om |u_\ve|^2 \big(1-|u_\ve|^2 \big)^2 dx +\limsup_{\ve \to 0} \int_\Om \big| \nabla (1-|u_\ve|^2 )\big|^2 dx \le C  \delta
\]
Letting $\delta  \to 0$, we obtain that
\begin{equation}\label{eq:(1-u^2)^2 vs (1-u^2)^3}
\begin{aligned}
0&= \lim_{\ve \to 0} \frac{1}{\ve^2} \int_\Om |u_\ve|^2 \big(1-|u_\ve|^2 \big)^2 dx \\
&  = \lim_{\ve \to 0} \frac{1}{\ve^2} \int_\Om  \big(1-|u_\ve|^2 \big)^2 dx - \lim_{\ve \to 0} \frac{1}{\ve^2} \int_\Om   \big(1-|u_\ve|^2 \big)^3 dx.
\end{aligned}
\end{equation}
and
\begin{equation}\label{eq:nabla 1-u^2}
\lim_{\ve \to 0} \int_\Om \big| \nabla (1-|u_\ve|^2 )\big|^2 dx =0.
\end{equation}
By using  \eqref{potential-bded}, \eqref{eq:nabla 1-u^2} and the Gagliardo-Nirenberg inequality
\[ \|u\|_3^3~ \le ~C ~ \|u\|_2^2 ~ \|\nabla u\|_2 \qfor u \in H^1_0(\Om),
\]
we are led to
\begin{align*}
& \frac{1}{\ve^2} \int_\Om   \big(1-|u_\ve|^2 \big)^3 dx \\
\le ~& \frac{C}{\ve^2} \left( \int_\Om   \big(1-|u_\ve|^2 \big)^2 dx \right)  \left( \int_\Om   \big| \nabla (1-|u_\ve|^2 ) \big|^2 dx \right)^{\frac12} \\
\le ~& C \ga_0   \left( \int_\Om  \big| \nabla (1-|u_\ve|^2 ) \big|^2 dx \right)^{\frac12} ~\to ~0.
\end{align*}
In the sequel, we conclude from \eqref{eq:(1-u^2)^2 vs (1-u^2)^3} that
\begin{equation}\label{eq:pot to zero in pf}
\lim_{\ve \to 0} \frac{1}{\ve^2} \int_\Om  \big(1-|u_\ve|^2 \big)^2 dx = \lim_{\ve \to 0} \frac{1}{\ve^2} \int_\Om   \big(1-|u_\ve|^2 \big)^3 dx = 0,
\end{equation}
which implies by Lemma \ref{lem:pot to zero implies unif conv} that   $|u_\ve | \to 1$ uniformly on $\overline\Om$.
This finishes the proof.
\end{proof}
%%%%%%%%%%%%%%%%%%%%%%%%%%%%%%%%%%%%%%

\smallskip
%%%%%%%%%%%%%%
{\bf Proof of Theorem \ref{thm:main1} (ii): }
Let us assume the contrary.
Then,  $|u_\ve| \to 1 $ uniformly on $\overline\Om $.
Hence,   \eqref{eq:limsup} holds by Lemma \ref{lem:limsup} which contradicts \eqref{G bdd by C1-C2}.
\qed  \\
%%%%%%%%%%%%%%%%%%%%%%%%%%%%%%%%%%
\smallskip

\smallskip
%%%%%%%%%%%%%%
{\bf Proof of Theorem \ref{thm:main2}: }
Suppose that \eqref{eq:potential to zero} holds.
Then, $|u_\ve | \to 1$ uniformly on $\Om$ by Lemma \ref{lem:pot to zero implies unif conv}.
Moreover, by Corollary \ref{cor:liminf} and Lemma \ref{lem:limsup}, we have
\[ \lim_{\ve \to \infty} \int_\Om  |\nabla u_\ve |^2  dx = \int_\Om |\nabla u_0 |^2  dx.
\]
Since $u_\ve \to u_0$ weakly in $H^1(\Om)$, we deduce from \eqref{eq:strong conv u to u0} that   $u_\ve \to u_0$ in $H_g^1(\Om)$.

Conversely, suppose that \eqref{eq:lim nabla u} is true.
Since $|u_\ve| \to 1$ uniformly on $\overline\Om$ by Lemma \ref{lem:H1 conv implies unif conv}, we may assume that $|u_\ve|^2 \ge 1/2$ and use notations in Lemma \ref{lem:unif conv} and Lemma \ref{lem:limsup}.
Multiplying \eqref{eq-rho-psi-2} by $\rho_\ve -1$, we obtain
\begin{align*}
& \int_\Omega |\nabla \rho_\varepsilon|^2  dx +\frac{1}{\ve^2}  \int_\Omega \rho_\ve(1- \rho_\varepsilon  )(1- \rho_\ve^2) \,dx  \\
=~& \int_\Omega (\rho_\ve - \rho_\ve^2) | \nabla (\varphi + \psi_\varepsilon )|^2 dx ~\le ~ \| 1- \rho_\ve \|_\infty \int_\Omega   | \nabla (\varphi + \psi_\varepsilon )|^2 dx ~\to ~ 0.
\end{align*}
Here, we used the fact that $u_\ve \to u_0$ in $H^1(\Om)$ such that $ \| \nabla (\varphi + \psi_\varepsilon )\|_2$ is bounded as $\ve \to 0$.
As a consequence,
\begin{align*}
 0&=\lim_{\ve \to 0} \frac{1}{\ve^2}  \int_\Omega \rho_\ve(1- \rho_\varepsilon  )(1- \rho_\ve^2) \,dx=
\lim_{\ve \to 0} \frac{1}{\ve^2}   \int_\Omega \frac{\rho_\ve}{1+\rho_\ve}  (1- \rho_\ve^2)^2dx \\
&=
\lim_{\ve \to 0} \frac{1}{2\ve^2} \int_\Om (1-\rho_\ve^2)^2 dx = 0
\end{align*}
and the proof is complete.
\qed
%%%%%%%%%%%%%%%%%%%%%%%%%%%%%%%%%%

 \bigskip
%%%%%%%%%%%%%%%%%%%%%%%%%%%%%%%%%%
%%%%%%%%%%%%%%%%%%%%%%%%%%%%%%%%%
%%%%%%%%%%%%%%%%%%%%%%%%%%%%%%%%%%%
\section{Proof of Theorem \ref{thm:main3}}\label{sec:thm pf main3}

Throughout this section, we assume \eqref{eq:deg zero g1 g2} and prove Theorem \ref{thm:main3}.
We also assume that   $\Om$ is starshaped. 
We can write
\[ g_1 = e^{i \vp_0} \qand g_2= e^{i \psi_0} \qwhere \vp_0, \psi_0:  \pa \Om \to \rone.
\]
The functions $u_0$ and $v_0$ are  lifted by   harmonic functions $\varphi$ and $\psi$ respectively  such that
\begin{equation}\label{eq:phi sys}
\left\{
\begin{aligned}
\Delta \vp & =0 \qin \Om \qand \vp=\vp_0 \qon \pa \Om,\\
u_0 &= e^{i \varphi} \qand \int_\Omega |\nabla u_0|^2 dx =\int_\Omega |\nabla \vp|^2 dx,
 \end{aligned}
 \right.
 \end{equation}
 and
\begin{equation}\label{eq:psi sys}
\left\{
\begin{aligned}
\Delta \psi & =0 \qin \Om \qand \psi=\psi_0 \qon \pa \Om,\\
v_0 &= e^{i \psi} \qand \int_\Omega |\nabla v_0|^2 dx =\int_\Omega |\nabla \psi|^2 dx.
 \end{aligned}
 \right.
 \end{equation}

 \bigskip
%%%%%%%%%%%%%%
{\bf Proof of Theorem \ref{thm:main3} (i): }
Suppose that \eqref{G < C_1-sys-nonsym} is valid.
Since $ \| u_\ve \|_\infty + \| v_\ve \|_\infty \leq 3$ by Lemma \ref{lem:L-infty variant} (ii), up to a subsequence, we have
$(u_\ve,v_\ve ) \rightharpoonup (\tilde{u},\tilv) $ in $H^1 (\Omega) \times H^1(\Omega)$ for some $(\tilu ,\tilv) \in H^1_{g_1}(\Omega) \times  H^1_{g_2}(\Omega)$.
By \eqref{eq:Pohozaev sys nonsym},  $|\tilu | =1$ and  $|\tilv | =1$ a.e. on $\Om$ and consequently $ {\tilde{ u}} \in H^1_{g_1}(\Om; S^1)$ and $ {\tilde{ v}} \in H^1_{g_2} (\Om; S^1)$.
Since $(u_0,v_0)$ is a unique minimizer of $I_{(g_1,g_2)}$ on $\calY(g_1,g_2)$, we are led to 
\begin{align*}
\frac{1}{2} \int_\Omega  \big( |\nabla u_0|^2+|\nabla v_0 |^2 \big) dx & \leq \frac{1}{2} \int_\Omega  \big( |\nabla \tilu|^2+|\nabla \tilv |^2 \big)\,dx \\
& \leq \liminf_{\ve \to 0}  \frac{1}{2}\int_\Omega  \big( |\nabla u_\ve|^2+|\nabla v_\ve |^2 \big) \,dx \\
&\leq C_3 \le
\frac{1}{2} \int_\Omega \big( |\nabla u_0|^2 + |\nabla v_0|^2 \big)\, dx  .
\end{align*} 
Thus, \eqref{eq:C1-1-sys-nonsym}  is true.
Moreover,   $u_\ve \to u_0$ in $H_{g_1}^1(\Om)$ and  $v_\ve \to v_0$ in $H_{g_2}^1(\Om)$ as in the proof of Theorem \ref{thm:main1} (i).
\qed \\
%%%%%%%%%%%%%%%%%%%%%%%%%%%%%%%%%%%%%%%%%5
\smallskip

 \smallskip
%%%%%%%%%%%%%%
{\bf Proof of Theorem \ref{thm:main3} (ii): }
Let us assume the contrary so that
\begin{equation}\label{eq:potential zero nonsym pf}
\lim_{\ve \to 0} \frac{1}{\ve^2}  \int_\Om  \Big[ \big( 2- |u_\ve|^2 - |v_\ve|^2 \big)^2 +    \big( 1- |u_\ve|^2   \big)^2 \Big] dx = 0.
\end{equation}
If \eqref{eq:potential zero nonsym pf} is valid, then it follows from \cite[Lemma 2.5]{HHS-nonsym-deg-0} that $|u_\ve|\to 1$ and $|v_\ve|\to 1$ uniformly on $\overline\Om$.
So, we may assume that $|u_\ve|^2 \ge 1/2$ and $|v_\ve|^2 \ge 1/2$ on $\Om$.
We can write
\begin{equation}\label{eq:zeta xi}
 u_\varepsilon = \rho_\varepsilon e^{i \zeta_\varepsilon} \qand
v_\varepsilon = \sigma_\varepsilon e^{i \xi_\varepsilon},
\end{equation}
where $\rho_\ve = |u_\ve|$ and $\sigma_\ve = |v_\ve|$.
Set
\begin{equation}\label{eq:eta chi}
\eta_\varepsilon = \zeta_\varepsilon   -\varphi \qand \chi_\varepsilon = \xi_\varepsilon - \psi.
\end{equation}
Then, \eqref{eq:2-GL nonsym} is written as
\begin{align}
\label{eq:rho-nonsys}
\operatorname{div} \big(\rho_\varepsilon^2 \nabla( \varphi+\eta_\varepsilon )\big) &= 0,  \\
\label{eq:zeta-nonsys}
-\Delta \rho_\varepsilon + \rho_\varepsilon | \nabla \varphi + \nabla \eta_\varepsilon |^2
&= \frac{1}{\varepsilon^2}\rho_\varepsilon \left( 2 -  \rho_\varepsilon^2  - \sigma_\varepsilon^2\right) + \frac{1}{\varepsilon^2}\rho_\varepsilon \left( 1 -   \rho_\varepsilon^2 \right),\\
\label{eq:sigma-nonsys}
\operatorname{div}\big(\sigma_\varepsilon^2 \nabla( \psi+\chi_\varepsilon )\big) &= 0, \\
\label{eq:xi-nonsys}
-\Delta \sigma_\varepsilon + \sigma_\varepsilon |  \nabla \psi+\chi_\varepsilon |^2
&= \frac{1}{\varepsilon^2}\sigma_\varepsilon \left( 2 -  \rho_\varepsilon^2  - \sigma_\varepsilon^2\right) .
\end{align}
By multiplying \eqref{eq:zeta-nonsys} by $\rho_\ve-1$ and \eqref{eq:xi-nonsys} by $\sigma_\ve-1$, we obtain from \eqref{G bdd by C1-C2-sys-nonsym}
\begin{align*}
C_4 ~ \le ~& \frac12 \int_\Om \big( |\nabla u_\ve|^2 + |\nabla v_\ve|^2 \big) \, dx \\
=~ &
\frac{1}{2}\int_\Omega
\Big(
|\nabla \rho_\varepsilon|^2 + |\nabla \sigma_\varepsilon|^2
+ \rho_\varepsilon^2 | \nabla \varphi+\nabla \eta_\varepsilon |^2
+ \sigma_\varepsilon^2 | \nabla \psi+\nabla \chi_\varepsilon |^2
\Big)  \\
= ~& \frac12 \int_\Omega  \Big( \rho_\varepsilon  | \nabla \varphi+\nabla \eta_\varepsilon|^2
+ \sigma_\varepsilon  | \nabla \psi+\nabla \chi_\varepsilon |^2\Big) \, dx + D_1+D_2+D_3,
\end{align*}
 where
 \begin{equation}\label{eq:D123}
 \left\{
 \begin{aligned}
 D_1 & =  \frac{1}{\varepsilon^2} \int \rho_\varepsilon (\rho_\varepsilon - 1)(2 - \rho_\varepsilon^2 -  \sigma_\varepsilon^2),\\
 D_2 & =  \frac{1}{\varepsilon^2} \int \sigma_\varepsilon (\sigma_\varepsilon - 1)(2 - \rho_\varepsilon^2 -   \sigma_\varepsilon^2),\\
  D_3& =  \frac{1}{\varepsilon^2} \int \rho_\varepsilon (\rho_\varepsilon - 1)(1 - \rho_\varepsilon^2).
 \end{aligned}\right.
 \end{equation}
 Then, $D_j\to 0$ for each $j=1,2,3$ as $\ve \to0$.
 Indeed,  by H\"{o}lder's inequality and the condition \eqref{eq:potential zero nonsym pf}, we can show that $D_1 \to 0$ and $D_3 \to 0$ as $\ve \to 0$.
 Moreover, as $\ve \to 0$, we have
\begin{align*}
o(1)&= \frac{1}{\varepsilon^2}  \int_\Om (2 - \rho_\varepsilon^2 -  \sigma_\varepsilon^2)^2 dx\\
&= \frac{1}{\varepsilon^2}\int_\Om  (1 - \rho_\varepsilon^2)^2 dx +  \frac{2}{\varepsilon^2}\int_\Om  (1 - \rho_\varepsilon^2)(1 - \sigma_\varepsilon^2) dx +  \frac{1}{\varepsilon^2}\int_\Om  (1 - \sigma_\varepsilon^2)^2 dx \\
& =o(1) +  \frac{2}{\varepsilon^2}\int_\Om  (1 - \rho_\varepsilon^2)(1 - \sigma_\varepsilon^2) dx +  \frac{1}{\varepsilon^2}\int_\Om  (1 - \sigma_\varepsilon^2)^2 dx.
\end{align*}
Hence, by H\"{o}lder's inequality, we obtain
\[  \frac{1}{\varepsilon^2}\int_\Om  (1 - \sigma_\varepsilon^2)^2 dx \le o(1)+
2 \left[ \frac{1}{\varepsilon^2}\int_\Om  (1 - \rho_\varepsilon^2)^2 dx \right]^{\frac12}
 \left[ \frac{1}{\varepsilon^2}\int_\Om  (1 - \sigma_\varepsilon^2)^2 dx \right]^{\frac12}.
\]
Thus, $\|1-\sigma_\ve^2 \|_2\to 0$ and  then H\"{o}lder's inequality implies that $D_2 \to 0$.

We have shown that as $\ve \to 0$,
\begin{equation}\label{eq:A1 A2 nonsym}
\begin{aligned}
C_4 ~&\le ~  o(1) + \frac{1}{2}\int_\Omega  \rho_\varepsilon  | \nabla \varphi+\nabla \eta_\varepsilon|^2 dx + \frac{1}{2}\int_\Omega
 \sigma_\varepsilon  |\nabla \psi+\nabla \chi_\varepsilon |^2 dx \\
&=:~ o(1)+A_1+A_2.
\end{aligned}
\end{equation}
Let     $\delta \in (0,\frac14)$ be given and we may assume \eqref{eq:delta chosen}.
So, we have
\begin{align*}
 A_1&  \le \frac{1}{2}\int_\Omega  \rho_\varepsilon^2   | \nabla \varphi+\nabla \eta_\varepsilon|^2 dx +  \frac{\delta}{2}\int_\Omega     | \nabla \varphi+\nabla \eta_\varepsilon|^2 dx\\
 &=  \frac{1}{2}\int_\Omega  \rho_\varepsilon^2   |\nabla \vp|^2 dx +   \frac{1}{2}\int_\Omega  \rho_\varepsilon^2  \big( 2  \nabla \vp \cdot \nabla \eta_\varepsilon +    |\nabla \eta_\varepsilon |^2 \big) \,dx +  \frac{\delta}{2}\int_\Omega     |\nabla \vp +\nabla \eta_\ve|^2 dx.
\end{align*}
By multiplying \eqref{eq:rho-nonsys} by $\psi_\ve$, we obtain
\begin{equation}\label{eq:vp psi sys nonsym}
\int_\Om \rho_\ve^2 |\nabla \eta_\ve|^2 dx+  \int_\Om \rho_\ve^2 \nabla  \vp \cdot \nabla \eta_\ve \,dx=0.
\end{equation}
So,
\begin{equation}\label{eq:A1 ineq 3}
A_1    \le  \frac{1}{2}\int_\Omega  \rho_\varepsilon^2   |\nabla \vp|^2 dx -   \frac{1}{2}\int_\Omega  \rho_\varepsilon^2 |\nabla \eta_\ve|^2dx   +  \frac{\delta}{2}\int_\Omega     | \nabla \vp +\nabla \eta_\varepsilon|^2 dx.
\end{equation}
On the other hand, \eqref{eq:vp psi sys nonsym} implies that
\begin{align*}
\int_\Om \rho_\ve^2 |\nabla \vp +\nabla \eta_\ve|^2 dx &= \int_\Om \rho_\ve^2 \big(\nabla \vp +\nabla \eta_\ve \big) \cdot \nabla \vp  \, dx \\
&\le  \left( \int_\Omega  \rho_\varepsilon^2  |\nabla \vp+\nabla \eta_\ve|^2dx \right)^{\frac12}  \left( \int_\Omega  \rho_\varepsilon^2  |\nabla \vp|^2dx \right)^{\frac12}.
\end{align*}
Since $\rho_\ve^2\ge 1/2$, this inequality implies that
\[ \frac12  \int_\Om |\nabla \vp +\nabla \eta_\ve|^2 dx \le \int_\Om \rho_\ve^2 |\nabla \vp +\nabla \eta_\ve|^2 dx \le \int_\Omega  \rho_\varepsilon^2  |\nabla \vp|^2dx.
\]
Hence, we can rewrite \eqref{eq:A1 ineq 3} as
\begin{align*}
 A_1 \le  \frac{1}{2}\int_\Omega  \rho_\varepsilon^2   |\nabla \vp|^2 dx  +    \delta \int_\Omega  \rho_\varepsilon^2    |\nabla \vp|^2 dx.
\end{align*}
By a similar argument, we also obtain
\begin{align*}
 A_2 \le  \frac{1}{2}\int_\Omega  \sigma_\varepsilon^2   |\nabla \psi|^2 dx  +   \delta  \int_\Omega     \sigma_\varepsilon^2 |\nabla \psi|^2 dx.
\end{align*}
In the sequel, we deduce from   \eqref{eq:A1 A2 nonsym} that
\[  C_4  \le o(1) +  \frac{1}{2}\int_\Omega  \big( \rho_\varepsilon^2   |\nabla \vp|^2+\sigma_\ve^2  |\nabla \psi|^2 \big) \, dx +  \delta \int_\Omega  \big( \rho_\ve^2  |\nabla \vp|^2 + \sigma_\ve^2 |\nabla \psi|^2 \big) \,dx.
\]
Letting $\ve \to 0$, we are led to
\[ C_4 \le \frac{1}{2}\int_\Omega  \big(    |\nabla \vp|^2+   |\nabla \psi|^2 \big) \, dx + \frac{ \delta}{2}\int_\Omega  \big(    |\nabla \vp|^2 +   |\nabla \psi|^2 \big) \,dx
\]
Finally, by taking the limit $\delta \to 0$, we get a contradiction from   the assumption \eqref{G bdd by C1-C2-sys-nonsym}.
\qed
%%%%%%%%%%%%%%%%%%%%%%%%%%%%%%%%%%%%
\smallskip

 \bigskip
%%%%%%%%%%%%%%%%%%%%%%%%%%%%%%%%%%
%%%%%%%%%%%%%%%%%%%%%%%%%%%%%%%%%
%%%%%%%%%%%%%%%%%%%%%%%%%%%%%%%%%%%
\section{Proof of Theorem \ref{thm:main4}}\label{sec:thm pf main4}

This section is devoted to the proof of  Theorem \ref{thm:main4}.
Throughout this section, we assume that \eqref{eq:deg zero g1 g2} holds and $\Om$ is starshaped.\\

 \smallskip
%%%%%%%%%%%%%%
{\bf Proof of Theorem \ref{thm:main4} (i):}
Suppose that \eqref{G < C_1-sys-nonsym} is valid.
Since $ \| u_\ve \|_\infty + \| v_\ve \|_\infty \leq 2$ by Lemma \ref{lem:L-infty variant} (i), up to a subsequence, we have
$(u_\ve,v_\ve ) \rightharpoonup (\tilde{u},\tilv) $ in $H^1 (\Omega) \times H^1(\Omega)$ for some $(\tilu ,\tilv) \in H^1_{g_1}(\Omega) \times  H^1_{g_2}(\Omega)$.
By \eqref{eq:Pohozaev sys sym},  $|\tilu |^2+|\tilv |^2 =2$ a.e. on $\Om$ and thus $ (\tilu,\tilv) \in \calX(g_1,g_2)$.
Since $(u_*,v_*)$ is a   minimizer of $I_{(g_1,g_2)}$ on $\calX(g_1,g_2)$, we are led to 
\begin{align*}
\frac{1}{2} \int_\Omega  \big( |\nabla u_*|^2+|\nabla v_* |^2 \big) dx & \leq \frac{1}{2} \int_\Omega  \big( |\nabla \tilu|^2+|\nabla \tilv |^2 \big)\,dx \\
& \leq \liminf_{\ve \to 0}  \frac{1}{2}\int_\Omega  \big( |\nabla u_\ve|^2+|\nabla v_\ve |^2 \big) \,dx \\
&\leq C_5 \le
\frac{1}{2} \int_\Omega \big( |\nabla u_*|^2 + |\nabla v_*|^2 \big)\, dx  .
\end{align*}
Thus, \eqref{eq:C1-1-sys-sym}  is obtained.
As in the proof of Theorem \ref{thm:main1} (i), it also holds that
   $u_\ve \to \tilu$ in $H_{g_1}^1(\Om)$ and  $v_\ve \to \tilv$ in $H_{g_2}^1(\Om)$.
Furthermore, if $\al(g_1,g_2)=\beta(g_1,g_2)$, then it is easy to see that $(u_*,v_*)=(\tilu,\tilv)=(u_0,v_0)$.
This completes the proof.
\qed
%%%%%%%%%%%%%%%%%%%%%%%%%%%%%%%%%%%%%%%%%5
\smallskip

\smallskip
%%%%%%%%%%%%%%%%%%%%%%%%%%%%%%%%%%%%%%%%%
\begin{remark}\label{rmk:tilu vs u*}
We do not know the uniqueness of solution to the problem \eqref{eq:min prob beta}.
If this problem has a unique solution, then we obtain $(u_*,v_*)=(\tilu,\tilv)$ in the proof of Theorem \ref{thm:main4} (i).
\end{remark}
%%%%%%%%%%%%%%%%%%%%%%%%%%%%%%%%%%%%%%%%%%%%%%%
\smallskip

 \smallskip
%%%%%%%%%%%%%%
{\bf Proof of Theorem \ref{thm:main4} (ii): }
Let us assume the contrary so that $|u_\ve|\to 1$ and $|v_\ve|\to 1$ uniformly on $\overline\Om$.
Then, $|u_*|=1$ and $|v_*|=1$.
Since $\al(g_1,g_2)=\beta (g_1,g_2)$ by \eqref{eq:al=beta}, it follows that $(u_*,v_*)=(u_0,v_0)$.
So, we can use the notations \eqref{eq:phi sys} and \eqref{eq:psi sys}.
Moreover,   we may assume that $|u_\ve|^2 \ge 1/2$ and $|v_\ve|^2 \ge 1/2$ on $\Om$, and take the notations \eqref{eq:zeta xi} and \eqref{eq:eta chi}.
We can rewrite  \eqref{eq:2-GL sym}as
\begin{align}
\label{eq:rho-sys}
\operatorname{div} \big(\rho_\varepsilon^2 \nabla( \varphi+\eta_\varepsilon )\big) &= 0,  \\
\label{eq:zeta-sys}
-\Delta \rho_\varepsilon + \rho_\varepsilon | \nabla \varphi + \nabla \eta_\varepsilon |^2
&= \frac{1}{\varepsilon^2}\rho_\varepsilon \left( 2 -  \rho_\varepsilon^2  - \sigma_\varepsilon^2\right)  ,\\
\label{eq:sigma-sys}
\operatorname{div}\big(\sigma_\varepsilon^2 \nabla( \psi+\chi_\varepsilon )\big) &= 0, \\
\label{eq:xi-sys}
-\Delta \sigma_\varepsilon + \sigma_\varepsilon |  \nabla \psi+\chi_\varepsilon |^2
&= \frac{1}{\varepsilon^2}\sigma_\varepsilon \left( 2 -  \rho_\varepsilon^2  - \sigma_\varepsilon^2\right) .
\end{align}
By proceeding as in the proof of Theorem \ref{thm:main3} (ii), we obtain
\begin{align*}
C_6 ~ \le ~&   \frac12 \int_\Omega  \Big( \rho_\varepsilon  | \nabla \varphi+\nabla \eta_\varepsilon|^2
+ \sigma_\varepsilon  | \nabla \psi+\nabla \chi_\varepsilon |^2\Big) \, dx + D_1+D_2 ,
\end{align*}
 where $D_1$ and $D_2$ are defined by \eqref{eq:D123}.
By \eqref{eq:Pohozaev sys sym}, \eqref{eq:gamma3}, \eqref{eq:gamma4} and H\"{o}lder's inequality, we obtain
\begin{align*}
 D_1 & =  \frac{1}{\varepsilon^2} \int_\Om  \frac{\rho_\ve}{\rho_\ve+1} (\rho_\varepsilon^2 - 1)(2 - \rho_\varepsilon^2 -  \sigma_\varepsilon^2) \le \sqrt{\ga_1\ga_3},\\
 D_2 & =  \frac{1}{\varepsilon^2} \int_\Om  \frac{\sigma_\ve}{\sigma_\ve+1} (\sigma_\varepsilon^2 - 1)(2 - \rho_\varepsilon^2 -  \sigma_\varepsilon^2) \le \sqrt{\ga_1\ga_4}.
\end{align*}
So,
\[ C_6 ~ \le ~   \frac{1}{2}\int_\Omega  \rho_\varepsilon  | \nabla \varphi+\nabla \eta_\varepsilon|^2 dx + \frac{1}{2}\int_\Omega
 \sigma_\varepsilon  |\nabla \psi+\nabla \chi_\varepsilon |^2 dx + \sqrt{\ga_1\ga_3} + \sqrt{\ga_1\ga_4}.
 \]
 Furthermore, by arguing as in the proof of Theorem \ref{thm:main3} (ii), we are led to
\[ C_6 \le \frac{1}{2}\int_\Omega  \big(    |\nabla \vp|^2+   |\nabla \psi|^2 \big) \, dx  + \sqrt{\ga_1\ga_3} + \sqrt{\ga_1\ga_4},
\]
we  contradicts   the assumption \eqref{G bdd by C1-C2-sys-sym}.
\qed
%%%%%%%%%%%%%%%%%%%%%%%%%%%%%%%%%%%%
\smallskip

\bigskip
%%%%%%%%%%%%%%%%%%%%%%%%%%%%%%%%%%
%%%%%%%%%%%%%%%%%%%%%%%%%%%%%%%%%%%%%
%%%%%%%%%%%%%%%%%%%%%%%%%%%%%%%%%%%%%
 \subsubsection*{\bf Acknowledgements.}
 Jongmin Han  was supported by the National Research Foundation of Korea(NRF) grant funded by the Korea government(MSIT) (RS-2024-00357675).

\small
%%%%%%%%%%%%%%%%%%%%%%%%%%%%%%%%%%%%%
%%%%%%%%%%%%%%%%%%%%%%%%%%%%%%%%%%%%%%%%%%%%
 \bibliographystyle{amsplain}

\end{document}